\documentclass[12pt, A4paper]{amsproc}
\usepackage{amsmath, amsthm, amssymb, booktabs, float}

\usepackage{booktabs}
\usepackage{url, hyperref}
\usepackage{anysize}
\usepackage{enumitem}
\usepackage{tikz}

\usepackage{listings}
\usepackage{mathrsfs}

 \newtheorem{theorem}{Theorem}[section]
 \newtheorem{lemma}[theorem]{Lemma}
 \newtheorem{corollary}[theorem]{Corollary}

\theoremstyle{remark}

\newtheorem{remark}[theorem]{Remark}

\renewcommand\le{\leqslant}
\renewcommand\ge{\geqslant}

\newcounter{followUpMarkerA}
\newcounter{followUpMarkerB}
\def\followups{ }
\newcommand{\followupAdd}[1]{\expandafter\def\expandafter\followups\expandafter{\followups{}\item \stepcounter{followUpMarkerB} Page \pageref{followupmarker\thefollowUpMarkerB}: #1}}

\marginsize{2cm}{2cm}{2cm}{2.5cm}
\hypersetup{citecolor=blue, linkcolor=blue, colorlinks=true}

\title[A family of hemisystems on the parabolic quadrics]{A family of hemisystems on the parabolic quadrics}
\author{Jesse Lansdown}
\author{Alice C. Niemeyer}

\email{Jesse.Lansdown@research.uwa.edu.au}
\email{Alice.Niemeyer@mathb.rwth-aachen.de}

\address[Lansdown]{Centre for the Mathematics of Symmetry of Computation\\
School of Mathematics and Statistics\\
University of Western Australia\\
Perth WA\\
Australia}
\address[Lansdown, Niemeyer]{
Lehrstuhl B f\"ur Mathematik\\
Lehr- und Forschungsgebiet Algebra\\
RWTH Aachen\\
Pontdriesch 10-16\\
52062 Aachen\\
Germany}

\thanks{The first author acknowledges the support of an Australian Government Research Training Program Scholarship and a UWA Top-Up Scholarship. The second author acknowledges the support the Australian Research Council Discovery Grant DP190100450. We would also like to thank John Bamberg and Michael Giudici for their helpful discussions. The foundation of this research occurred during the first author's time at the RWTH}

\begin{document}
\maketitle

\begin{abstract}
We constuct a family of hemisystems of the parabolic quadric $\mathcal{Q}(2d, q)$, for all ranks $d \ge 2$ and all odd prime powers $q$, that admit $\Omega_3(q) \cong \mathrm{PSL}_2(q)$. This yields the first known construction for $d \ge 4$.
\end{abstract}

\section{Introduction}

Let $\mathcal{S} = (\mathcal{P}, \mathcal{M}, I)$ be an incidence structure with \emph{points} $\mathcal{P}$ and \emph{maximals} $\mathcal{M}$. We say that $\mathcal{S}$ has order $(s,t)$ if there are $s+1$ points incident with every maximal, and $t+1$ maximals incident with every point. A \emph{hemisystem} is a subset $H$ of $\mathcal{M}$ such that every point is incident with $\frac{t+1}{2}$ maximals of $H$ (thus requiring that $t$ is odd). A hemisystem $H$ is said to ``admit'' a group $B$ if $B$ is isomorphic to a subgroup of the stabiliser of $H$ in the automorphism group of $\mathcal{S}$.

Hemisystems have connections to other objects in geometry, graph theory, and coding theory. In particular, they often induce new objects such as partial quadrangles, strongly regular or distance regular graphs, and association schemes \cite{CameronDelsarteGoethals, Dam_Martin_Muzychuk_2013}.

A parabolic quadric of rank $d$ is constructed by taking the totally singular subspaces of a $(2d+1)$-dimensional vector space over $\mathbb{F}_q$ under a quadratic form, and is denoted by $\mathcal{Q}(2d, q)$. The parabolic quadrics will be defined more explicitly along with their connections to the orthogonal group in Section \ref{Background}.

Vanhove showed that an $\frac{s+1}{2}$-ovoid of certain dual polar spaces yield new distance regular graphs with classical parameters \cite{Vanhove_2011}, and in
\cite{BambergLansdownLee}
it was shown that the only $m$-ovoids of $\mathcal{DQ}(2d, q)$, $\mathcal{DH}(2d-1, q^2)$, and $\mathcal{DW}(2d-1, q)$, for $d\ge 3$, are $\frac{s+1}{2}$-ovoids. Moreover, an $\frac{s+1}{2}$-ovoid in $\mathcal{DQ}(2d, q)$, $\mathcal{DH}(2d-1, q^2)$, or $\mathcal{DW}(2d-1, q)$ is a hemisystem of $\mathcal{Q}(2d, q)$, $\mathcal{H}(2d-1, q^2)$, or $\mathcal{W}(2d-1, q)$, respectively. 
Recently, Cossidente and Pavese found an infinite family of hemisystems of $\mathcal{Q}(6,q)$, $q$ odd, admitting $\mathrm{PSL}_2(q^2)$ \cite{COSSIDENTE2016112}. This is currently the only known family of hemisystems of the parabolic quadrics, for $d \ge 3$. For $d=2$, hemisystem constructions have been found by Feng et al. \cite{FENG2016} as well as by Cossidente et al. \cite{Cossidente_Culbert_Ebert_Marino_2008}. Hence prior to Theorem \ref{main} of this paper no families of hemisystems of $\mathcal{Q}(2d,q)$ were known for $d \ge 4$.

In this paper we construct a new infinite family of hemisystems of $\mathcal{Q}(2d, q)$ for $d \geq 2$ and $q$ odd.

\begin{theorem}\label{main}
There exist $2^n$ hemisystems of $\mathcal{Q}(2d, q)$ admitting $\mathrm{\Omega}_3(q) \cong \mathrm{PSL}_2(q)$ for all odd prime powers $q$ and all $d \ge 2$, where $n$ is the number of orbits of $\mathrm{\Omega}_3(q)$ on the maximals.
\end{theorem}

Hemisystems were first defined by Segre on Hermitian varieties, where he demonstrated the existence of a hemisystem in $\mathcal{H}(3, 3^2)$,
and raised the question
whether they exist in $\mathcal{H}(2d-1, q^2)$ for $d >2, q > 3$
\cite{Segre1965}. 
For a long time no new examples were found, and it was thought that Segre's example might be the only example, with Thas even conjecturing that there were no hemisystems of $\mathcal{H}(3,q^2)$, for $q>3$ \cite{ThasConjecture}.  This conjecture was disproved, however, when Cossidente and Penttila constructed infinite families for $\mathcal{H}(3,q^2)$ \cite{cossidente_penttila_2005} and $\mathcal{H}(5, q^2)$ \cite{Cossidente2009}, $q$ odd. In his PhD thesis, Luke Bayens constructed hemisystems of $\mathcal{H}(2d - 1, q)$, $d > 2, q$ odd \cite{Bayens_2013}, thus answering Segre's question. His construction introduced the so-called ``AB-Lemma'', which is also utilised by the construction in this paper, and is elaborated upon in Section \ref{sec:ABLemma}.

Hemisystems have also been generalised beyond Hermitian varieties. Cameron, Goethals, and Seidel extended the definition of a hemisystem to a generalised quadrangle of order $(q, q^2)$, for $q$ odd, and showed that the collinearity graph of such a hemisystem is strongly regular \cite{CameronDelsarteGoethals, Cameron_Goethals_Seidel_1978}.

Vanhove extended the concept of a hemisystem to regular near polygons, in particular showing that in the dual hermitian space $\mathcal{DH}(2d-1, q^2)$, for $q$ odd and $d\ge 3$, the existence of a hemisystem would induce new distance regular graphs with classical parameters \cite{Vanhove_2011}. 

Bamberg, Guidici, and Royle showed that every flock generalized quadrangle of order $(s^2,s)$, $s$ odd, contains a hemisystem \cite{Bamberg_Giudici_Royle_2010}, and van Dam, Martin, and Muzychuk showed that hemisystems of generalised quadrangles of order $(s^2, s)$ give rise to $4$-class cometric association schemes \cite{Dam_Martin_Muzychuk_2013}.

A common approach to the construction of geometric objects is to consider a subgroup of the automorphism group, and to stitch together its orbits on the elements of the geometry. Since elements of one type interact with an element of another type in the same manner within an orbit, far fewer elements need then be considered. This approach lends itself to large subgroups of the automorphism group, since this means there are fewer orbits, making it is easier to consider the interplay between them. By contrast, the hemisystems in the family presented in this paper admit a small group relative to the full automorphism group of the parabolic quadric. In fact, the admitted group is dependent only on $q$ and is constant regardless of the rank of the quadric. We construct the hemisystems by considering a parabolic plane in the ambient projective space and consider how the points and maximals of the parabolic quadric meet this plane.

\section{Background}\label{Background}

In this section we cover the necessary background theory required to prove Theorem \ref{main}. Most are standard definitions and results in the subject, and can be found in, for example, \cite{Aschbacher, KleidmanLiebeck}.

\subsection{Vector spaces with a quadratic form}
Let $V$ be an $n$-dimensional vector space over $\mathbb{F}_q$.
With respect to a basis $\{ e_1, e_2, \ldots, e_n \}$, the \emph{Gram matrix} $J$ of a bilinear form $\beta$ is the $n \times n$ matrix with entries
\[
J_{ij} = \beta(e_i, e_j).
\]
The Gram matrix describes the bilinear form with respect to the given basis, where
\[
\beta(v, w) = v J w^T.
\]

A \emph{quadratic form on $V$} is a map $\kappa: V \to \mathbb{F}_q$ such that for all $v \in V$ and all $\lambda \in \mathbb{F}_q$
\[
\kappa(\lambda v) = \lambda^2 \kappa(v)
\]
and $\beta(u, v) := \kappa(u + v) - \kappa(u) - \kappa(v)$ defines a bilinear form $\beta$ on $V$, called the \emph{associated bilinear form}. Note that $\kappa(v) = \beta(v, v)/2$, so the bilinear form uniquely determines $\kappa$ when $\mathbb{F}_q$ is of odd characteristic; that is, when $q$ is an odd prime power. In this paper we are concerned only with $q$ odd, so we may work with the bilinear and quadratic forms interchangeably, depending on which better suits the task at hand. 

A bilinear form $\beta$ is \emph{degenerate} if there exists some $v \in V\backslash \{0 \}$ such that $\beta(u, v)= 0$ for all $u \in V$, and \emph{nondegenerate} otherwise. A quadratic form is degenerate if its associated bilinear form is degenerate. A subspace $U \leq V$ is \emph{totally singular} if $\kappa(u)=0$ for all $u \in U$, and \emph{anisotropic} if $\kappa(u)\neq 0$ for all non-zero vectors in $U$. The \emph{Witt index} of a vector space equipped with a quadratic form is the dimension of the largest totally singular subspace. The \emph{perp} of a subspace $U$ is defined as $U^\perp := \{v \in V \mid \beta(u, v) = 0, \forall u \in U \}$. If $U$ is nondegenerate, then so too is $U^\perp$. If $u, v \in V$ are two vectors such that $\beta(u, v) = 1$ and $\beta(u, u) = \beta(v, v)=0$, then $( u, v )$ is called a \emph{hyperbolic pair}.

Up to equivalence, there are just three types of vector spaces with nondegenerate quadratic forms: \emph{parabolic, hyperbolic}, and \emph{elliptic}. The parabolic case occurs when $n$ is odd, while hyperbolic and elliptic cases occur for $n$ even. Moreover, $V$ has an orthogonal decomposition $V= H_1 \oplus H_2 \oplus \ldots \oplus H_d \oplus X$, where each $H_i$ is the span of a hyperbolic pair, $X$ is an anisotropic subspace with  $\dim(X)=0$ in the hyperbolic case, $\dim(X)=1$ in the parabolic case, and $\dim(X)=2$ in the elliptic case. The Witt index is given by $d$ in the previous decomposition, and hence the Witt index of a parabolic space is $\frac{1}{2}(n-1)$, the Witt index of a hyperbolic space is $\frac{n}{2}$, and the Witt index of an elliptic space is $\frac{n}{2}-1$.

Hyperbolic forms are often referred to as ``$+$'' type, while elliptic forms are referred to as ``$-$'' type, providing an easy notation to distinguish between the two cases in even dimension. Since there is only one parabolic form in odd dimension it is unnecessary to indicate its type, however we may refer to it as ``$\circ$'' type for consistency, or to identify it when the dimension is not explicitly stated.

\subsection{The quadrics} Let $V$ be an $n$-dimensional vector space over $\mathbb{F}_q$ equipped with a nondegenerate quadratic form $\kappa$.
Taking the totally singular subspaces of $V$ we obtain a \emph{parabolic quadric $\mathcal{Q}(2d, q)$}, a \emph{hyperbolic quadric} $\mathcal{Q}^+(2d-1,q)$, or an \emph{elliptic quadric} $\mathcal{Q}^-(2d+1, q)$, corresponding to the type of the form on $V$. Here $d$ is the Witt index and $2d$, $2d-1$, and $2d+1$ each give $n-1$, which is the projective dimension. The totally singular $1$-spaces are called points, while the largest totally singular subspaces are called \emph{maximals}. Incidence is then defined as inclusion of subspaces. The \emph{rank} of the quadric is given by the Witt index and corresponds to the number of different types of elements in the geometry. 

Recall that the order of an incidence structure is $(s,t)$ where there are $s+1$ points on each maximal, and $t+1$ maximals on every point. Thus a hemisystem may exist only when $t+1$ is even.
In the case of $\mathcal{Q}(2d,q)$,
\[
s+1 = \frac{q^{d}-1}{q-1}, \quad \text{ and }\quad t+1 = \prod\limits_{i=1}^{d-1}(q^i + 1).
\]
Hence we may limit ourselves to the case where $q$ is odd, since $t+1$ above is even precisely when $q$ is odd. Moreover, a hemisystem contains half of the set of maximals, and the complement of a hemisystem is also a hemisystem.

The automorphism group of the quadric is the group which preserves the totally singular subspaces, $\mathrm{P\Gamma O}_n^\epsilon(q)$. We elaborate on groups preserving quadratic forms in the following section.

\subsection{Groups preserving the quadratic form}
Let $V$ be a $n$-dimensional vector space over $\mathbb{F}_q$ equipped with a nondegenerate quadratic form $\kappa$. Let $q$ be an odd prime power throughout.

The subgroup of $\mathrm{GL}(V)$ preserving the form is called the \emph{orthogonal group}, denoted $\mathrm{O}(V)$.
The \emph{special orthogonal group}, $\mathrm{SO}(V)$, is the subgroup of $\mathrm{O}(V)$ consisting of the elements with determinant one, and the derived subgroup of $\mathrm{O}(V)$ is denoted by $\mathrm{\Omega}(V)$. Each group is an index two subgroup of the previous group, that is, $|\mathrm{O}(V) : \mathrm{SO}(V)| = |\mathrm{SO}(V) : \mathrm{\Omega}(V)| = 2$ \cite[Table 2.1.C]{KleidmanLiebeck}.

Moreover, since 
the nondegenerate quadratic forms are unique up to a change of basis, we may write $\mathrm{O}_n^\epsilon(q)$, $\mathrm{SO}_n^\epsilon(q)$, and $\mathrm{\Omega}_n^\epsilon(q)$, where $\epsilon \in \{ +, \circ, - \}$ corresponds to the type of the form, $n$ is the dimension of the vector space, and $q$ is the order of the field. We will often still write $\mathrm{O}(V)$, $\mathrm{SO}(V)$, or $\mathrm{\Omega}(V)$, to emphasise the vector space $V$. In particular, we adapt the notation to apply to a subspace of $V$ to mean the image of the projection of the group onto the subspace. Formally, let $X$ be $\mathrm{O}$, $\mathrm{SO}$, or $\mathrm{\Omega}$, and let $W$ be a nondegenerate subspace of dimension $m$ in $V$, where $\kappa|_W$ is the restriction of $\kappa$ to $W$, then
\begin{equation} \label{defnX}
X(W) := \{ g \oplus 1_{W^\perp} \mid g \in X_m^{\epsilon'}(q) \text{ with respect to } \kappa|_W \},
\end{equation}
for some $\epsilon' \in \{ +, \circ, -\}$.
Note that the restriction of $\kappa$ to $W$ need not have the same type as $\kappa$ itself, and hence $\epsilon$ is not necessarily equal to $\epsilon'$, for $X(V) \cong X_n^\epsilon(q)$. Moreover, $X(W)$ need not be a subgroup of $X$.

There also exist projective versions of each of these groups,
\begin{align*}
\mathrm{PX}(V) &:= \mathrm{X}(V) / (\mathrm{X}(V) \cap Z(\mathrm{GL}(V))),
\end{align*}
for $X = \mathrm{O}, \mathrm{SO},$ or $\mathrm{\Omega}$.
For a vector space over a field, $Z(GL(V))$ is simply all the non-zero scalar matrices, and so the projective versions of the groups are the original groups modulo the corresponding scalar matrices. As a result the projective versions of the groups act naturally on one dimensional subspaces rather than on vectors.

More on the classical groups can be found in Kleidman and Liebeck \cite{KleidmanLiebeck}. A few results which form part of more general results in Kleidman and Liebeck are collected here. Recall that $q$ is assumed to be odd.
\begin{lemma}\cite[2.9.1]{KleidmanLiebeck}\label{lem:isomorphisms}
There exist the following isomorphisms:

\begin{enumerate}
	\item $\mathrm{PSL}_2(q) \cong \mathrm{\Omega}_3(q)$,
	\item $\mathrm{O}_2^\pm(q) \cong D_{2(q \mp 1) }$,
	\item $\mathrm{SO}_2^\pm(q) \cong \mathbb{Z}_{q \mp 1}$,
	\item $\mathrm{\Omega}_2^\pm(q) \cong \mathbb{Z}_{(q\mp 1)/2}$.
\end{enumerate}
\end{lemma}

The vectors of $V$ can be partitioned according to their value under the quadratic form $\kappa$, so for $\alpha \in \mathbb{F}$ we define
\begin{equation} \label{Valpha}
V_\alpha := \{v \in V\backslash \{0\} \mid \kappa(v) = \alpha \}.
\end{equation}
We have the following orbit results on $V_\alpha$.

\begin{lemma}\cite[2.10.5]{KleidmanLiebeck}\label{OmegaOrbits}$\quad$
\begin{enumerate}
	\item $\mathrm{O}_n^\epsilon(q)$ is transitive on $V_\alpha$, for all $n$, $\alpha$, and $\epsilon$.
	\item $\mathrm{\Omega}^\circ_3(q)$ has two orbits on $V_0$ of size $\frac{1}{2}(q^2-1)$ and is transitive on $V_\alpha$ for $\alpha \neq 0$.
	\item $\mathrm{\Omega}_2^+(q)$ has $4$ orbits on $V_0$, and $\mathrm{\Omega}_2^\pm(q)$ has $2$ orbits on $V_\alpha$ for $\alpha \neq 0$.
\end{enumerate}
\end{lemma}

We denote the stabiliser in $H \le \mathrm{O}(V)$ of a subspace $W$ or a vector $v$, by $H_W$ or $H_v$, respectively. For a subgroup $H$ fixing a subspace $W$, the subgroup $H$ induces upon $W$ is denoted by $H^W$.

The following lemma describes how the orthogonal group interacts with the stabilisers of a nondegenerate subspace and its perp. It holds in more generality, but for our purposes we restrict it to $\mathrm{O}(V)$. 

\begin{lemma} \cite[4.1.1]{KleidmanLiebeck}\label{OmegaStabiliser}
Assume that $V = U \perp W$, where $U$ is nondegenerate, and $X = \mathrm{O}, \mathrm{SO}$, or $\mathrm{\Omega}$. Then:
\begin{enumerate}
	\item	$\mathrm{O}(V)_{U} = \mathrm{O}(U) \times \mathrm{O}(W)$,
	\item	$\mathrm{\Omega}(V)_{U} \geq \mathrm{\Omega}(U) \times \mathrm{\Omega}(W)$,
	\item	$X(U) \cap \mathrm{\Omega}(V) = \mathrm{\Omega}(U)$ and $X(W) \cap \mathrm{\Omega}(V) = \mathrm{\Omega}(W)$,
	\item	$\mathrm{\Omega}(V)_{U}^{U} = \mathrm{O}(U)$,
	\item $\mathrm{\Omega}(V)_{U}^{W} = \mathrm{O}(W)$, unless $\dim(U)=1$.
\end{enumerate}
\end{lemma}

When considering the subspace spanned by a hyperbolic pair, it is easy to describe the elements of the orthogonal group explicitly. Since the special orthogonal group consists of the determinant one elements of the orthogonal group and the derived subgroup has index two in the special orthogonal group, the elements of these groups are also easily describable.

\begin{lemma}\label{lem:OmegaElements} Given the quadratic form $\kappa(x_1, x_2)= x_1x_2$ for a two dimensional vector space over $\mathbb{F}_q$,
\[
\mathrm{O}^+_2(q) = 
\bigg
\{
\left(
\begin{array}{cc}
 \gamma &0  \\
0 & \gamma^{-1}  \\
\end{array}
\right), \left(
\begin{array}{cc}
 0 & \gamma \\
\gamma^{-1} & 0  \\
\end{array}
\right)
: \gamma \in \mathbb{F}_q^*
\bigg \}.
\]
Moreover, $\mathrm{SO}_2^+(q)$ consists only of the diagonal elements, and $\mathrm{\Omega}_2^+(q)$ consists of those elements with squares on the diagonal.
\end{lemma}

We summarise some of the core information relating to a vector space $V$ equipped with a quadratic form $\kappa$ in Table \ref{tab:summary} below.

\begin{table}[h!]
\begin{center}
\begin{tabular}{ lclclclcl } 
 \hline
 Type & $n = \dim(V)$ & $\epsilon$ & $d =$ Witt Index & $O(V)$ \\
\hline
 Parabolic & Odd & $\circ$ & $\frac{1}{2}(n-1)$ & $\mathrm{O}^\circ_n(q)$  \\
 Hyperbolic & Even & $+$ & $\frac{n}{2}$ & $\mathrm{O}_n^+(q)$ \\
 Elliptic & Even & $-$ & $\frac{n}{2}-1$ & $\mathrm{O}_n^-(q)$ \\
 \hline
\end{tabular}
\end{center}
\caption{Vector spaces with quadratic forms}
\label{tab:summary}
\end{table}

\subsection{The $AB$-Lemma} \label{sec:ABLemma}

The following lemma, often referred to as the ``AB-Lemma'', was first stated in Luke Bayens' dissertation \cite{Bayens_2013}. Given an incidence geometry whose automorphism group contains subgroups with certain properties, the lemma helps prove the existence of hemisystems without the need to construct the tactical configuration.

\begin{lemma}[The AB-Lemma {\cite[4.4.1]{Bayens_2013}}]\label{lem:AB} Let $\mathcal{S} = (\mathcal{P}, \mathcal{M}, I)$ be an incidence structure with two types, called points $\mathcal{P}$ and maximals $\mathcal{M}$.
Let $A$ and $B$ be two subgroups of the automorphism group of $\mathcal{S}$ such that
\begin{enumerate}
	\item $B$ is a normal subgroup of $A$,
	\item $A$ and $B$ have the same orbits on $\mathcal{P}$,
	\item each $A$-orbit on $\mathcal{M}$ splits into two $B$-orbits.
\end{enumerate}
Then there are $2^m$ hemisystems admitting $B$, where $m$ is the number of $A$-orbits on the maximals.
\end{lemma}

Moreover, a hemisystem can be constructed by taking a representative from each of the $A$-orbits and then taking the union of the orbits of these representatives under the action of $B$.

The following section is dedicated to the proof of the main theorem. We construct two subgroups of the orthogonal group and show that they satisfy the conditions of the AB-Lemma.

\section{Main result}
This section is dedicated to proving Theorem \ref{main} by applying the $AB$-Lemma to a suitable construction of the parabolic quadric.

Let $V$ be a $(2d+1)$-dimensional vector space over $\mathbb{F}_q$, for $q$ odd, $d \ge 2$, with a basis 
\[
\mathcal{B}= 
\{z, e_0, f_0, x, y, e_1, f_1, e_2, f_2 \ldots, e_{d-2}, f_{d-2} \}.
\]

Let $\beta$ be a nondegenerate bilinear form defined on $V$ such that 
\[
V = \langle z \rangle \perp \langle e_0, f_0 \rangle \perp \langle x, y \rangle \perp \langle e_1, f_1 \rangle \ldots \langle e_{d-2}, f_{d-2}\rangle
\]
where $( e_i, f_i )$ for $i \in \{0, \ldots, d-2\}$ are hyperbolic pairs and $\langle z \rangle$ and $\langle x, y \rangle$ are anisotropic subspaces. In particular let $\beta(e_i, f_i)=1$, $\beta(e_i, e_i) = \beta(f_i, f_i)=0$, and $\beta(z, z)=1$.
Let $\kappa$ be the associated quadratic form with $\beta$, defined by $\kappa(v)= \beta(v, v)/2$.

We distinguish two subspaces $W = \langle z, e_0, f_0 \rangle$ and $U = W^\perp = \langle x, y, e_1, f_1, \ldots, e_{d-2}, f_{d-2} \rangle$. The construction makes it clear that $W$ is a $3$-dimensional parabolic subspace of $V$, while $U$ is a $(2d-2)$-dimensional elliptic subspace of $V$. 

The automorphism group of $\mathcal{Q}(2d, q)$ is $\mathrm{P\Gamma O}_{2d+1}(q)$, however it is sufficient for us to consider the matrix group $\mathrm{O}_{2d+1}(q)$ as its action on subspaces is the same as that of $\mathrm{PO}_{2d+1}(q) \le \mathrm{P\Gamma O}_{2d+1}(q)$.
Thus we let $G=\mathrm{O}_{2d+1}(q)$.

Let $B = \Omega(W)$, where we recall $\Omega(W)$ is defined as in (\ref{defnX}). Moreover, take $\tau \in G$, where

\begin{equation} \label{tau}
\tau = \left(
\begin{array}{ccc|c}
 -1 & & &  \\
& 1 & & 0\\
& & 1 & \\
\hline
& 0 &  & I_{2d-2}   \\
\end{array}
\right),
\end{equation}
and $I_{2d-2}$ is the identity matrix. And lastly let
\[
  A = \langle  B,  \tau \rangle.
\]

We observe that $\Omega(W) \le A \le O(W)$ and so every element of $A$ (and also of $B$) has the form
\[
\left(
\begin{array}{c|c}
 g &  0  \\
\hline
0 & I_{2d-2}   \\
\end{array}
\right)
\]
for some $g \in O_3^\circ(q)$ with respect to $\kappa |_W$. Since each element of $A$ (and therefore of $B$) contains this identity block, the subgroups $A$ and $B$ are unchanged modulo scalar matrices: $A / (A \cap Z) \cong A$.

We make the following observation,
\begin{remark}
$B \cong PSL_2(q)$, since $B \cong \Omega^\circ_3(q) \cong PSL_2(q)$ by Lemma \ref{lem:isomorphisms}(1).
\end{remark}
We now give a few technical lemmas to aid in later proofs.

\begin{lemma}\label{lem:StabiliserOfNonDegenVector}
Let $T$ be a $3$-dimensional vector space.
Let $v$ be a non-singular vector of $T$, and let $T_1 = \langle v \rangle$, $T_2 = \langle v \rangle^\perp$ such that $T = T_1 \perp T_2$, then
\[
\mathrm{\Omega}(T_2) \le \mathrm{\Omega}(T)_{T_1} = \mathrm{\Omega}(T)_{T_2} \le \mathrm{O}(T_1) \times \mathrm{O}(T_2)
\]
and
\[
\mathrm{\Omega}(T)_v \cong \mathrm{\Omega}(T_2).
\]
\end{lemma}

\begin{proof}
Since $v$ is non-singular, $T_1$ and $T_2$ are nondegenerate, thus by Lemma \ref{OmegaStabiliser}(2) and (1), $\mathrm{\Omega}(T_1) \times \mathrm{\Omega}(T_2)  \le \mathrm{\Omega}(T)_{T_2} \le \mathrm{O}(T)_{T_2} =  \mathrm{O}(T_1) \times \mathrm{O}(T_2)$. Since $\dim(T_1)=1$, $\mathrm{O}(T_1) = \{\pm 1 \}$ and $\mathrm{SO}(T_1) = \mathrm{\Omega}(T_1) = 1$. Fixing $T_2$ means also fixing $T_1$, so $\mathrm{\Omega}(T)_{T_1} = \mathrm{\Omega}(T)_{T_2}$, and the first part follows. If in addition, $v$ is fixed, then $\mathrm{\Omega}(T)_v^{T_1} = 1$, and so  $ \mathrm{\Omega}(T_2)   \le \mathrm{\Omega}(T)_v \le \mathrm{O}(T_2)$. Now by Lemma \ref{OmegaStabiliser}(3), it follows that $\mathrm{O}(T_2) \cap \mathrm{\Omega}(T) =  \mathrm{\Omega}(T_2)$, so the second part follows.
\end{proof}

Note that $T_2$ in Lemma \ref{lem:StabiliserOfNonDegenVector} is a nondegenerate two dimensional subspace and hence $\mathrm{O}(T_2)$ and $\mathrm{\Omega}(T_2)$ could be of either $+$ or $-$ type, depending on $T_2$ itself.

Observe that $W$ is a $3$-dimensional subspace of $V$ with the same properties as $T$ in Lemma \ref{lem:StabiliserOfNonDegenVector}.
In particular, the stabiliser in $\mathrm{\Omega}(W)$ of a nonsingular vector $v$ is isomorphic to $\mathrm{\Omega}_2^\pm(q)$. Hence it is useful to know exactly how subgroups of $\mathrm{O}_2^\pm(q)$ act. Recall the definition of $V_\alpha$ in (\ref{Valpha}).

\begin{lemma}\label{ReflectionTypes}
Let $\Delta_1$ and $\Delta_2$ be the two orbits of $\mathrm{\Omega}_2^\epsilon(q)$ on $V_\alpha$, for $\alpha \neq 0$, and let $g \in \mathrm{O}^\epsilon_2(q)\backslash \mathrm{SO}^\epsilon_2(q)$. Then $\Delta_i^g = \Delta^{\hphantom{g}}_i$ precisely when $g$ fixes a point of $V_\alpha$, and $\Delta_i^g = \Delta^{\hphantom{g}}_j$ otherwise, for $i \in \{1,2 \}$, $j \in \{1,2\} \backslash \{i\}$.
\end{lemma}

\begin{proof}
By Lemma \ref{lem:isomorphisms}, $\mathrm{O}^\pm_2(q)$ is dihedral of order $2(q \mp 1)$, while $\mathrm{SO}^\pm_2(q)$ and $\mathrm{\Omega}^\pm_2(q)$ are cyclic. Moreover, by Lemma \ref{OmegaOrbits}, $\mathrm{O}^\pm_2(q)$ is transitive on $V_\alpha$, while $\mathrm{\Omega}^\pm_2(q)$ has two orbits, for $\alpha \neq 0$.
Now, $|\mathrm{O}^\pm_2(q):\mathrm{SO}^\pm_2(q)|=2$ and $|\mathrm{SO}^\pm_2(q): \mathrm{\Omega}^\pm_2(q)|=2$, and so $\mathrm{O}_2^\pm(q)$ has the following subgroup structure
\begin{center}
\begin{tikzpicture}[scale=.7]
  \node (one) at (0,2) {$\mathrm{O}_2^\pm(q) = \langle s, t \colon |s|= q\mp 1, \quad |t|= 2, \quad tst = s^{-1} \rangle$};
  \node (a) at (-3,0) {$\mathrm{SO}_2^\pm(q) = \langle s \rangle$};
  \node (b) at (0,0) {$\langle s^2, t \rangle$};
  \node (c) at (3,0) {$\langle s^2, st \rangle$};
  \node (zero) at (0,-2) {$\mathrm{\Omega}_2^\pm(q) = \langle s^2 \rangle$};
  \draw (zero) -- (a) -- (one) -- (b) -- (zero) -- (c)  -- (one);
\end{tikzpicture}\\
\end{center}

Note that the action of $\mathrm{O}_2^\pm(q)$ on $V_\alpha$ is permutation isomorphic to the action of $D_{2(q\mp1)}$ on the vertices of a regular $(q\mp1)$-gon (which has an even number of vertices for $q$ odd), and moreover  $D_{2(q\mp1)}$ has two conjugacy classes of reflections: reflections in an axis through opposite vertices and reflections in an axis through opposite midpoints. 
The first type fix exactly two points and we shall call such reflections ``hyperbolic'', while the second fix no points and shall be called ``elliptic''.
Considering the action of the dihedral group, it is clear that a hyperbolic reflection preserves the orbits of the cyclic subgroup $\langle s^2 \rangle$ which rotates the vertices, while an elliptic reflection interchanges the orbits. Since $\mathrm{SO}_2^\pm(q)$ is cyclic, an element $\sigma \in \mathrm{O}_2^\pm(q) \backslash \mathrm{SO}_2^\pm(q)$ must be a reflection. The existence of a fixed point on $V_\alpha$ determines if $\sigma$ is hyperbolic or elliptic, and hence determines its action on the orbits of $\mathrm{\Omega}_2^\pm(q)$ on $V_\alpha$.
\end{proof}

A key observation arising from the previous lemma is that although $O_2^\pm(q)$ is permutationally isomorphic to $D_{2(q\mp1)}$, a specific reflection $\sigma$ may be hyperbolic in its action upon $V_\alpha$, but elliptic in its action upon $V_{\alpha'}$, for $\alpha \neq \alpha'$ and $\alpha, \alpha' \neq 0$.

\begin{lemma}\label{lem:BnormalinA}
$B$ is a proper normal subgroup of $A$.
\end{lemma}

\begin{proof}
Note that the derived subgroup of a group is a normal, hence $B = \mathrm{\Omega}(W) \trianglelefteq \mathrm{O}(W)$, and thus $B \trianglelefteq A$. Moreover, $\det( \tau)=-1$, so $ \tau \not \in  B$ and $ B < A$. It then follows that $B \triangleleft A$.
\end{proof}

We now consider the action of the two subgroups $A$ and $B$ on the points and maximals of $\mathcal{Q}(2d, q)$. We recall that the elements of $\mathcal{Q}(2d, q)$ are totally singular subspaces, and that $A$ and $B$ act on vectors of $W$ (not just subspaces).
Note that $V = W \perp U$, hence every $v \in V$ can be expressed as $v = w + u$ for some $w \in W$ and $u \in U$. Moreover, $A$ and $B$ fix every vector in $U$, so it is sufficient to consider the action of $A$ and $B$ on $w$ to investigate their action on $v$. We recall the definition of $\tau$ in (\ref{tau}).

\begin{lemma}\label{lem:pointsfixed}
Let $P$ be a point of $\mathcal{Q}(2d, q)$. Then there exists some $g \in B$ such that $P^g = P^\tau$.
\end{lemma}

\begin{proof}
$P = \langle p \rangle$ for some $p \in V$, such that $p = w + u$ for $w \in W$ and $u \in U$. Now, $p^\tau = w^\tau + u$, since $A$ fixes $u$, and hence we require $g \in B$ such that $w^g = w^\tau$.

Let $w = \gamma_1 z + \gamma_2 e_0 + \gamma_3 f_0$. If $\gamma_1 = 0$ then $w^{1_G} = w^\tau$, where ${1_G} \in B$ is the identity element. If $\gamma_1 \neq 0$, then without loss of generality we may take $\gamma_1 = 1$, since $P = \langle p \rangle = \langle \gamma_1^{-1} p \rangle$.

Consider now $\gamma_1=1$ and  $\kappa|_W(w) = \alpha$.
If $\alpha \neq 0$ then by Lemma \ref{OmegaOrbits}(2) there exists some $g \in B$ such that $w^{g} = w^{\tau}$, and hence $p^g = p^\tau$ and $P^g = P^\tau$.

Consider instead $\alpha = 0$, then by Lemma \ref{OmegaOrbits}(2), there are two orbits on $W_\alpha$ under $B$. Set $v = \gamma_2 e_0 + \gamma_3 f_0$ and $v' = \gamma_2 e_0 - \gamma_3 f_0$. Then $\tau$ fixes $v$, and by Lemma \ref{lem:StabiliserOfNonDegenVector}, the stabiliser of $v$ in $B$ is $\Omega(\langle v \rangle^\perp)  = \Omega_2^\pm(q)$. Observe that $\tau$ fixes $v'$ in $\langle v \rangle^\perp$, but not $z$. Note that $p = z + v$, $\kappa(p)=0$, and $\kappa(z)=\frac{1}{2}$ imply that $\kappa(v)= \gamma_2 \gamma_3 = -\frac{1}{2}$, which in turn implies that $\kappa(v') = - \gamma_2 \gamma_3 = \frac{1}{2}$. Now $v', z \in \langle v \rangle^\perp$ and $v'$ is fixed by $\tau$, hence by Lemma \ref{ReflectionTypes} there exists some $g \in \Omega^\pm_2(q) \le  B$ such that $w^{g} = w^\tau$ and hence $p^g = p^\tau$ and $P^g = P^\tau$.
\end{proof}

\begin{corollary}
The orbits of $A$ and $B$ on the points are the same.
\end{corollary}

Note that we are not forced to fix $v$ in the previous proof, it simply proved convenient in demonstrating the existence of an appropriate group element $g$ in $B$. However, in forthcoming proofs we will seek to show that $\tau$ is unique in its action on maximals, and it will be necessary to fix certain subspaces according to the action of $\tau$.

\begin{lemma} \label{lem:NonEmptyProjection}
Let $M$ be a maximal totally singular subspace in $V$, and let $M'$ be a maximally totally singular subspace in $\langle z \rangle^\perp$. Then $M$ projects nontrivially onto $\langle z \rangle$ and $M'$ projects non-trivially onto $\langle e_0, f_0 \rangle$.
\end{lemma}

\begin{proof}
Consider $\langle z \rangle^\perp = \langle e_0, f_0 \rangle \perp \langle x, y \rangle \perp \langle e_1, f_1 \rangle \ldots \langle e_{d-1}, f_{d-1} \rangle$. Since $\langle e_i, f_i \rangle$ are hyperbolic planes, and $\langle x, y \rangle$ is an anisotropic subspace, $\langle z \rangle^\perp$ is an elliptic subspace under the form $\beta|_{\langle z \rangle^\perp}$. Since $\langle z \rangle^\perp$ has dimension $2d$ it has Witt index $d-1$. However a maximal $M$ in $V$ has dimension $d$, so $M$ cannot be entirely contained in $\langle z \rangle^\perp$. Thus there is a nonempty projection of $M$ onto $\langle z \rangle$.
Similarly, $U$ is an elliptic subspace of dimension $2d-2$, with Witt index $d-2$. Hence $M'$ is not completely contained in $U$ and must have a nonempty projection onto $\langle e_0, f_0 \rangle$.
\end{proof}

\begin{lemma}\label{lem:Basis}
Let $M$ be a maximal. There exists a basis $b_1, b_2, b_3, \ldots, b_d$ of $M$ such that
\[
b_1 = z + u_1, \quad b_2 = e_0 + u_2, \quad b_3 = f_0 + u_3, \quad b_i = u_i \quad \text{ for } i \in \{4, \ldots , d\},
\]
or
\[
b_1 = z + \lambda f_0 + u_1, \quad b_2 = e_0 + \mu f_0 + u_2, \quad b_i = u_i \quad \text{ for } i \in \{3, \ldots , d\},
\]
or
\[
b_1 = z + \lambda e_0 + u_1, \quad b_2 = f_0 + u_2, \quad b_i = u_i \quad \text{ for } i \in \{3, \ldots , d\},
\]
where in each case $u_i \in U$ for $i \in \{1, \ldots d\}$.
\end{lemma}

\begin{proof}
Let $b_1, b_2, \ldots, b_d$ be a basis of $M$. By Lemma \ref{lem:NonEmptyProjection} $M$ projects onto $\langle z \rangle$, and so without loss of generality, we may take $b_1$ to project non-trivially onto $\langle z \rangle$. Since $\dim(\langle z \rangle)=1$, there exists a basis $b_1, b_2', b_3', \ldots, b_d'$ of $M$, where each $b_i'$ projects trivially onto $\langle z \rangle$ for $i \in \{2, \ldots, d \}$. Such a basis can be obtained by taking linear combinations of $b_1$ from $b_i$. 

Now consider $M' = \langle b_2', \ldots, b_d' \}$, a $(d-1)$-dimensional totally isotropic subspace of $\langle z \rangle^\perp$. By Lemma \ref{lem:NonEmptyProjection} again, $M'$ projects non-trivially onto $\langle e_0, f_0 \rangle$. Without loss of generality, we may assume $b_2'$ has a non-empty projection to $\langle e_0, f_0 \rangle$. Since $\dim(\langle e_0, f_0 \rangle)=2$, we then have two cases: either there is some other basis vector, which we may take to be $b_3'$, such that $b_3'$ projects non-trivially onto $\langle e_0, f_0 \rangle$ and the projections of $b_2'$ and $b_3'$ onto $\langle e_0, f_0 \rangle$ are linearly independent, or else $b_2'$ is the only such vector. By taking linear combinations we may then manipulate the bases into the desired form.
\end{proof}

\begin{lemma}\label{lem:taudoubles}
	Let $M$ be a maximal. There does not exist $g \in B$ such that $M^g = M^\tau$.
\end{lemma}

\begin{proof}

Let $\{ b_1, b_2, \ldots, b_d \}$ be a basis for the totally singular subspace $M$. Without loss of generality, we need only consider the first two bases of Lemma \ref{lem:Basis}, since the argument for the third basis is identical to that of the second.
Now, $\tau$ fixes each basis vector other than $b_1$, and for any $g$ in $B$, $g$ fixes each vector in $U$. Hence $M^g = M^\tau$ if and only if $\langle b_1, b_2, b_3 \rangle^g = \langle b_1^\tau, b_2, b_3 \rangle$ in the first case, or $\langle b_1, b_2 \rangle^g  = \langle b_1^\tau, b_2  \rangle$ in the second case.
Throughout, we recall that $\kappa(b_i)=0$ and $\beta(b_i, b_j) = 0$, since $M$ is totally singular.
We now consider each case. 

{\it Case 1:} Consider $M$ with first three basis vectors
\[
b_1 = z + u_1, \quad b_2 = e_0 + u_2, \quad b_3 = f_0 + u_3
\]
for $u_1, u_2, u_3 \in U$.

Observe that $\kappa(b_1) = 0$, $\kappa(b_3) = 0$, and $\beta(b_1, b_3) = 0$ imply that $\beta(u_1, u_1) = -1$, $\kappa(u_3)=0$, and $\beta(u_1, u_3)=0$, respectively.

We first show that $u_2 \not \in \langle u_1, u_3 \rangle$. Consider for a contradiction that $u_2 = \gamma_1 u_1 + \gamma_2 u_3$. Then $2\kappa(u_2) = \beta(\gamma_1 u_1 + \gamma_2 u_3, \gamma_1 u_1 + \gamma_2 u_3) = -\gamma_1^2$. However $\kappa(b_2)=0$ implies that $\kappa(u_2)=0$, and thus $\gamma_1=0$ and $u_2 = \gamma_2 u_3$. Now, from $\beta(b_2, b_3)=0$ it follows that $\beta(u_2, u_3)=-1$ and therefore $\beta(\gamma_2 u_3 , u_3) = 2\gamma_2 \kappa(u_3) = -1$. However $\kappa(u_3) = 0$, a contradiction. So $u_2 \not \in \langle u_1, u_3 \rangle$.

Moreover, since $\kappa(u_1)=-\frac{1}{2}$ and $\kappa(u_3)=0$, it follows that $u_1 \neq \gamma u_3$. Hence $u_1, u_2$, and $u_3$ are linearly independent.

Let $g \in B$ such that
$b_1^g, b_2^g, b_3^g \in \langle b_1^\tau, b_2, b_3 \rangle$.
Then
\begin{align*}
b_1^g = z^g + u_1^g = z^g + u_1,
\end{align*}
and
\begin{align*}
b_1^g &= \gamma_1b_1^\tau + \gamma_2 b_2 + \gamma_3 b_3\\
&= -\gamma_1 z + \gamma_1 u_1 + \gamma_2 e_0 + \gamma_2 u_2 + \gamma_3 f_0 + \gamma_3 u_3.
\end{align*}
However, since $u_1, u_2$, and $u_3$ are linearly independent, $\gamma_2 = \gamma_3 = 0$ and $\gamma_1=1$. Hence $ b_1^g = -z + u_1 = b_1^\tau$.
Moreover, $b_2^g = b_2$ and $b_3^g = b_3$ by the same argument.
Thus
\[
z^g = -z, \quad
e_0^g  = e_0, \quad
f_0^g = f_0.
\]
From this it follows that $g =\tau$. However $ \tau \not \in  B$, so there is no such element $g \in B$ such that $M^g = M^\tau$.
\\

{\it Case 2:} Consider $M$ with first two basis vectors
\[
b_1 = z + \lambda f_0 + u_1, \quad b_2 = e_0 + \mu f_0 + u_2.
\]
Recall that $b_3 \in U$. From $\kappa(b_1)= 0$ if follows that $\kappa(u_1)=-\frac{1}{2}$, and since $\kappa(b_3) = 0$ it follows that $u_1$ and $b_3$ are linearly independent vectors in $U$.
Consider now $u_1 = \gamma u_2$. Since $u_1$ and $b_3$ are linearly independent, so too must $u_2$ and $b_3$ be linearly independent. Thus we may take a basis $b_1', b_2, b_3, \ldots, b_d$ of $M$, where $b_1' = b_1 + b_3 = z + \lambda f_0 + u_1'$ and $u_1' = u_1 + b_3$. We now have $u_1'$ is linearly independent of $u_2$, so without loss of generality, we may assume $u_1 \neq \gamma u_2$.

Let $g \in B$ such that $b_1^g, b_2^g \in \langle b_1^\tau, b_2 \rangle$. Then
\begin{align*}
b_1^g &= z^g + \lambda f_0^g + u_1^g = z^g + \lambda f_0^g + u_1,\\
b_2^g &= e_0^g + \mu f_0^g + u_2^g  = e_0^g + \mu f_0^g + u_2,
\end{align*}
and
\begin{align*}
b_1^g &= \gamma_1 b_1^\tau + \gamma_2 b_2 = \gamma_1 (-z + \lambda f_0 + u_1) + \gamma_2 (e_0 + \mu f_0 + u_2),\\
b_2^g &= \gamma_3 b_1^\tau + \gamma_4 b_2 = \gamma_3 (-z + \lambda f_0 + u_1) + \gamma_4 (e_0 + \mu f_0 + u_2).
\end{align*}
Since $u_1$ and $u_2$ are linearly independent, it follows that $\gamma_1 = \gamma_4 = 1$ and $\gamma_2 = \gamma_3 = 0	$. That is, $b_2^g = b_2$ and $b_1^g = b_1^\tau$. From this it follows, that,
\[
(z + \lambda f_0)^{g} = -z + \lambda f_0, \quad (e_0 + \mu f_0)^{g} = e_0 + \mu f_0.
\]

Here we have three subcases: $\mu = 0$, $2 \mu \in \square \backslash \{0\}$, and $2 \mu \not \in \square$, where $\square$ is the set of squares of $\mathbb{F}_q$.\\

{\it Case 2a:} Consider $\mu = 0$. Then $\kappa(b_2)=0$ implies $\kappa(u_2) = 0$. Moreover, $\beta(u_2, b_i)=0$ for $3 \le i \le d$. Observe, $b_3, \ldots, b_d$ are linearly independent, since they are basis vectors, and they are all contained in $U$. Since $U$ is of elliptic type, it has Witt index $d-2$, and hence $u_2$ is either a linear combination of $b_3, \ldots, b_d $ or $u_2=0$. In either case, $\beta(u_1, u_2)=0$, since $\beta(u_1, b_i)=0$ for $3 \le i \le d$ and $\beta(u_1, 0)=0$. However, $\beta(b_1, b_2)= 0$ and so $\beta(u_1, u_2)= -\lambda$, thus $\lambda =0$.

Now, we require $g$ such that $z^{g} = -z$ and $e_0^{g}=e_0$. It follows then that $g$ fixes $\langle z \rangle$ and hence $g \in B_{\langle z \rangle}$, where $1 \times \mathrm{\Omega}_2^+(q) \le B_{\langle z \rangle} \le \mathrm{O}_1^\circ(q) \times \mathrm{O}_2^+(q)$  by Lemma \ref{lem:StabiliserOfNonDegenVector}, since $\kappa(z) \neq 0$ and $\langle e_0, f_0 \rangle$ is of hyperbolic type. Moreover, $ g = (-1, h)$, where $h \in \mathrm{O}_2^+(q)$, since  $z^{ g} = -z$. However, $ g \in \mathrm{\Omega}^\circ_3(q) \le \mathrm{SO}^\circ_3(q)$, and hence $\det( g)=1$, meaning $\det(h) = -1$. By Lemma \ref{lem:OmegaElements}, $h$ must then be an element 
of the form 
\[
\begin{pmatrix} 0 & \gamma \\ \gamma^{-1} & 0 \end{pmatrix},
\]
however no such element fixes $e_0$.
Thus there is no such $ g \in  B$.\\

In the remaining two subcases where $\mu \neq 0$,
let $v = e_0 + \mu f_0$, and $\mathcal{C} = \{v, z, e_0 - \mu f_0 \}$ be a basis for $W$. Further let $W_1 = \langle v \rangle$ and $W_2 = \langle z, e_0 - \mu f_0 \rangle$.
Expressed with respect to $\mathcal{C}$, 
$z + \lambda f_0 = z -  \frac{\lambda}{2 \mu} (e_0 - \mu f_0) +  \frac{\lambda}{2 \mu} v$.
However $v^{ g} = v$, and so  $( z -  \frac{\lambda}{2 \mu} (e_0 - \mu f_0) )^{ g} = -z -  \frac{\lambda}{2 \mu} (e_0 - \mu f_0)$.
Note $g$ fixes $v$, hence $g \in B_v$. Moreover $\kappa(v) = \mu \neq 0$, so by Lemma \ref{lem:StabiliserOfNonDegenVector}, $B_v = \mathrm{\Omega}(W_2) \cong \mathrm{\Omega}_2^\epsilon(q)$.
\\

{\it Case 2b:} Consider $2 \mu \in \square \backslash \{0\}$. Then there exists $\gamma \neq 0$ such that $\gamma^2 = 2 \mu$.
Let $w = \gamma z + (e_0 - \mu f_0)$. Then $\kappa(w)=0$ and hence $W_2$ is hyperbolic, meaning $B_v \cong \mathrm{O}_2^+(q)$. Thus we require $ g = (1, h)$, where $h \in \mathrm{O}_2^+(q)$, such that  $(z - \frac{\lambda}{2\mu}(e_0 - \mu f_0) )^h = -z - \frac{\lambda}{2\mu}(e_0 - \mu f_0)$.
By Lemma \ref{lem:OmegaElements} $h$ has the form $\text{Diag}(\zeta^2, \zeta^{-2})$.
Now, $- \frac{\lambda}{2\mu}(e_0 - \mu f_0)^h = - \frac{\lambda}{2\mu}(e_0 - \mu f_0)$ implies $h =  \text{Diag}(1,1)$. However, $z^h =-z$ implies $h =  \text{Diag}(-1,-1)$, a contradiction.  
Thus there is no such $g \in B$ which replicates the action of $\tau$ on $M$.
\\

{\it Case 2c:} Consider $2 \mu \not \in \square$. Recall that $\mathrm{O}^\epsilon_2(q)$ is dihedral and $\mathrm{\Omega}^\epsilon_2(q)$ is cyclic, by Lemma \ref{lem:isomorphisms}.
With respect to $\mathcal{C}$, $\tau$ remains unchanged, and induces the element $\tau_{W_2}  = \text{Diag}(-1, 1) \in \mathrm{O}(W_2) \cong \mathrm{O}^\epsilon_2(q)$ when restricted to $W_2$. Clearly $\det(\tau_{W_2})=-1$ and hence $\tau_{W_2} \not \in \mathrm{\Omega}(W_2) = B_v$. Note that the determinant is unaffected by a change of basis of $W_2$. Moreover $\tau_{W_2}$ is not an element of $\mathrm{SO}(W_2)$ and is thus a reflection in the dihedral group $\mathrm{O}(W_2)$.
Observe that $\tau_{W_2}$ 
fixes only elements of the form $\gamma (e_0 - \mu f_0)$. Now, $\kappa(\gamma (e_0 - \mu f_0)) = -\mu\gamma^2  $, and so $ \tau_{W_2}$ fixes an element of $V_\alpha$ for $\alpha \in - \mu \square$ and is fixed-point-free otherwise.
Moreover $\lambda^2 \neq 2\mu$,
hence $\kappa(z - \frac{\lambda}{2\mu}(e_0 - \mu f_0)) \neq 0$ and so $z - \frac{\lambda}{2\mu}(e_0 - \mu f_0) \not \in V_0$.
Thus if $\kappa(z- \frac{\lambda}{2 \mu}(e_0 - \mu f_0)) = \frac{1}{2} + (\frac{\lambda}{2 \mu})^2(-\mu) \not  \in -\mu \square$, then  by Lemma \ref{ReflectionTypes},
$\tau_{W_2}$ interchanges the orbits of $\mathrm{\Omega}(W_2)$ on $V_{\frac{1}{2} + (\frac{\lambda}{2 \mu})^2(-\mu)}$, meaning $z- \frac{\lambda}{2 \mu}(e_0 - \mu f_0)$ and $-z- \frac{\lambda}{2 \mu}(e_0 - \mu f_0)$ lie in different orbits of $\mathrm{\Omega}(W_2) = B_v$.

We claim that indeed $\kappa(z- \frac{\lambda}{2 \mu}(e_0 - \mu f_0)) 
\not  \in -\mu \square$. To prove the claim, observe that 
$\frac{1}{2} + (\frac{\lambda}{2 \mu})^2(-\mu) \not \in - \mu \square$
if and only if
$\lambda^2 -2 \mu \not \in \square$.
Consider for a contradiction that $\lambda^2 -2 \mu \in \square$. Then there exists $\gamma$ such that $\kappa(\gamma u_1 + u_2)=0$, since 
$\kappa(\gamma u_1 + u_2)
= \gamma^2 \kappa(u_1) +  \gamma \beta(u_1, u_2) + \kappa(u_2)
= -\frac{\gamma^2}{2} - \gamma \lambda -\mu
$ has discriminant $\lambda^2 - 2 \mu$.
Hence $\langle \gamma u_1 + u_2, u_3, \ldots, u_d \rangle$ is a totally isotropic subspace in $U$ of dimension $d-1$. However, as $U$ is elliptic, its Witt index is $d-2$, a contradiction, thus proving the claim.
It follows then that there is no such $g \in B$ which acts on $M$ the same way as $\tau$.
\\

After considering each case, we see that there is no element $g \in B$ such that $M^g = M^\tau$.
\end{proof}

As a result of Lemma \ref{lem:taudoubles} the image of $\tau$ on each element in an orbit of $B$ on maximals is outside of the orbit, and since 
$B$ is index $2$ in $A$, this results in $B$ doubling each orbit of $A$. This gives the immediate corollary:

\begin{corollary}\label{cor:doubleorbits}
Each orbit of $A$ on maximals splits into two orbits of $B$ on maximals.
\end{corollary}

The proof of Theorem \ref{main} follows directly from Lemma \ref{lem:BnormalinA}, Lemma \ref{lem:pointsfixed}, Corollary \ref{cor:doubleorbits} and the application of the $AB$-Lemma, Lemma \ref{lem:AB}.

\section{Concluding Remarks}
In \cite[Remark 2.12]{COSSIDENTE2016112}, a sporadic example of a hemisystem of $\mathcal{Q}(6,3)$ is given which bares similarities to the construction in this paper. The authors provide a similar decomposition of the ambient space into a conic and its perp (of elliptic type). A subgroup isomorphic to $A_5$ is then found by computer in the pointwise stabiliser of the conic which fixes a hemisystem. It is not known if the construction in \cite{COSSIDENTE2016112} generalises to all $q$, or to greater rank, and so the authors describe it as sporadic.
Our construction instead finds a subgroup in the pointwise stabiliser of the perp of the conic which is isomorphic to $\mathrm{\Omega}_3(3) \cong A_4$, in the case $q=3$, and generalises to $\mathrm{\Omega}_3(q)$ for all odd $q$ and rank at least $2$.

We conclude with a more geometric description of the construction given in this paper. Let $p$ be a point of $PG(2d, q)$ such that $p^\perp$ is of elliptic type. Then $\tau$ is the unique involution which fixes $p^\perp$ point-wise. Let $\ell$ be any line of $p^\perp$ of hyperbolic type. Then $B$ is the derived subgroup of the point-wise stabiliser of $\langle p, \ell \rangle^\perp$. It is clear that $\langle p, \ell \rangle$ as a vector subspace corresponds to $W$ in Section 3.

\bibliographystyle{abbrv}
\bibliography{references}

\end{document}